\documentclass{amsart}

\usepackage{amssymb}
\usepackage{amsmath}
\usepackage{array}
\usepackage{graphicx}
\usepackage{caption}
\usepackage{subcaption}

\usepackage[small,nohug,heads=vee]{diagrams}
\diagramstyle[lablestyle=\scriptstyle]

\usepackage[small,nohug,heads=vee]{diagrams}
\diagramstyle[lablestyle=\scriptstyle]

\begingroup\makeatletter\ifx\SetFigFont\undefined
\def\x#1#2#3#4#5#6#7\relax{\def\x{#1#2#3#4#5#6}}%
\expandafter\x\fmtname xxxxxx\relax \def\y{splain}%
\ifx\x\y   
\gdef\SetFigFont#1#2#3{%
  \ifnum #1<17\tiny\else \ifnum #1<20\small\else
  \ifnum #1<24\normalsize\else \ifnum #1<29\large\else
  \ifnum #1<34\Large\else \ifnum #1<41\LARGE\else
     \huge\fi\fi\fi\fi\fi\fi
  \csname #3\endcsname}%
\else
\gdef\SetFigFont#1#2#3{\begingroup
  \count@#1\relax \ifnum 25<\count@\count@25\fi
  \def\x{\endgroup\@setsize\SetFigFont{#2pt}}%
  \expandafter\x
    \csname \romannumeral\the\count@ pt\expandafter\endcsname
    \csname @\romannumeral\the\count@ pt\endcsname
  \csname #3\endcsname}%
\fi
\fi\endgroup

\newcommand{\Pf}{{\em Proof}. }
\newcommand{\EPf}{\hfill$\square$}

\newcommand{\Lg}{\mbox{$\mathfrak g$}}

\newcommand{\Lh}{\mbox{$\mathfrak h$}}

\newcommand\liegr{\sf}

\newcommand{\SU}[1]{\mbox{${\liegr SU}(#1)$}}

\newcommand{\U}[1]{\mbox{${\liegr U}(#1)$}}
\newcommand{\SP}[1]{\mbox{${\liegr Sp}(#1)$}}
\newcommand{\SO}[1]{\mbox{${\liegr SO}(#1)$}}

\newcommand{\OG}[1]{\mbox{${\liegr O}(#1)$}}
\newcommand{\Spin}[1]{\mbox{${\liegr Spin}(#1)$}}
\newcommand{\G}{\mbox{${\liegr G}_2$}}
\newcommand{\F}{\mbox{${\liegr F}_4$}}
\newcommand{\E}[1]{\mbox{${\liegr E}_{#1}$}}

\newcommand\fieldsetc{\mathbb}

\newcommand{\R}{\fieldsetc{R}}
\newcommand{\C}{\fieldsetc{C}}
\newcommand{\Q}{\fieldsetc{H}}

\newtheorem{thm}{Theorem}[section]

\newtheorem{prop}[thm]{Proposition}
\newtheorem{lem}[thm]{Lemma}

\theoremstyle{remark}

\DeclareMathOperator{\Ima}{Im}

\title{Highly curved orbit spaces}

\author[C.~Gorodski]{Claudio Gorodski}

\address{Instituto de Matem\'atica e Estat\'\i stica, Universidade de
S\~ao Paulo, Rua do Mat\~ao, 1010, S\~ao Paulo, SP 05508-090, Brazil}

\email{gorodski@ime.usp.br}

\date{\today}

\begin{document}

\maketitle

\begin{abstract}
It is known that the infimum of the sectional curvatures (on the regular part) 
of orbit spaces
of isometric actions on unit spheres in bounded above by $4$. 
We show that the infimum is $1$ for ``most'' actions, and determine
the cases in which it is bigger than $1$. 
\end{abstract}

\section{Introduction}

Let $G$ be a compact Lie group acting non-transitively
by isometries on the unit sphere $S^n$, where $n\geq2$. 
The orbit space $X=S^n/G$ is an Alexandrov space of curvature at least $1$ 
and diameter at most $\pi$. Let  $\kappa_X$ denote the infimum of 
sectional curvatures of the quotient Riemannian metric in the regular part of 
$X$. It was
proved in~\cite{GL4} that $\kappa_X\leq4$ always holds. 
Note that $\kappa_X$ can also be characterized
as the largest number~$\kappa$ such that $X$ is an 
Alexandrov space of curvature $\geq\kappa$ (cf.~subsection~\ref{substrata}).

It is apparent from the discussion in~\cite{GL4} that 
$\kappa_X=1$ for ``most'' representations. This remark motivates
the present work. If $\kappa_X>1$, we will say that $X$ is 
\emph{highly curved}; 
we will also abuse language and say that $\rho$ is highly curved.  
Herein we prove the following theorem.

\begin{thm}\label{main}
Let $X=S^n/G$ be the orbit space of an isometric action of a compact Lie
group~$G$ on the unit sphere $S^n$ 
and assume that $\dim X\geq2$. Then $X$ is highly curved if and only if:
\begin{itemize}
\item[(i)] $n=7$ and $G^0=\U2$ acts on $S^7$ as the restriction of the 
irreducible representation on $\C^4$;
\end{itemize}
or the associated representation of $G^0$ is
quotient-equivalent to a non-polar (cf.~subsection~\ref{polar}) 
sum of two representations of
cohomogeneity one; in the latter case, 
either one of the following cases occur:
\begin{itemize}
\item[(ii)] $X$ is a good Riemannian orbifold of constant curvature $4$;
\item[(iii)] $X$ is a a complex weighted projective line 
(of real dimension two) or a $\mathbb Z_2$-quotient thereof;
\item[(iv)] $n=6$ and $G^0=\SU2$ acts on $S^6$ as the restriction of the representation $\C^2\oplus\R^3$;
\item[(v)] $G^0=\SP{m}\times\U1$ and it acts on $S^{4m-1}$, 
where $\SP m$ acts diagonally on $\C^{2m}\oplus\C^{2m}$ 
and $\U1$ acts with weights~$r$, $s\geq0$ where~$r\neq s$.
\end{itemize}
\end{thm}

There are two senses in which the curvature of $X$ is related to 
its diameter. First, the more extrinsically curved a $G$-orbit is, the closer
its focal points are, and thus the sooner a normal geodesic starting
there ceases to be minimizing. 
It is shown in~\cite{GSa} that the infimum over all
actions (coming from irreducible representations,
non-transitive on the unit sphere) 
of the supremum over all orbits of their focal radii
is bounded away from zero. Second, and more
relevant to this paper, the Bonnet-Myers argument implies
that $\mathrm{diam}\,X\leq\pi/\sqrt{\kappa_X}$. In~\cite{GLLM}
it is proved the existence of a (non-explicit) universal positive lower bound 
for $\mathrm{diam}\,X$ in case of a
nontransitive action on $S^n$ with $n\geq2$.  

There are families of actions for which $X$ is a Riemannian 
orbifold as in~(ii) and (iii), see \cite{GL3} for a classification.
In case~(iii), $X$ is a bad Riemannian orbifold and $\rho$ 
is quotient-equivalent to a circle action or a $\mathbb Z_2$-extension
thereof; a classification of the representations of maximal-connected
groups can be found in~\cite[Table~II, Types~$S^2$ and~$I$]{str}.    
The other cases do not yield Riemannian orbifolds. 
Case~(i) is the only example in the list in which the representation
is irreducible and reduced (cf.~subsection~\ref{types}); 
incidentally, this representation is not amenable to the 
general principles developed in this paper and its analysis requires
a direct calculation of the curvature involving the Thorpe method,
which we take up in subsection~\ref{thorpe}.
In general, the proof of Theorem~\ref{main} combines geometric and
algebraic arguments,
with analyses of special cases and use of representation theory
and several classification results. It would be very 
interesting to find simple geometric reasons why the representations 
listed in the theorem (and only them) are highly curved. 

The author wishes to thank Francisco J.~Gozzi and Pedro Z\"uhlke
for useful comments and Alexander Lytchak for several discussions, including
key ideas.



\section{Preliminaries}

\subsection{Spherical quotient spaces} \label{substrata}
Let $\rho:G\to\OG V$ be a representation of a compact 
Lie group on an Euclidean space $V$. It restricts to an isometric 
action on the unit sphere $S(V)$, and all isometric actions of 
compact Lie groups on unit spheres are obtained in this way. 
The cohomogeneity of $\rho$ coincides with $\dim(S(V)/G)+1$.
The quotient space
$X=S(V)/G$ is an Alexandrov space stratified by
smooth Riemannian manifolds, namely,
the projections of the sets of points in $S(V)$ with 
conjugate isotropy groups ---
the connected components of the strata can be equivalently
characterized as the connected components of the 
subsets of $X$ consisting of points with isometric tangent cones. 
There is a unique maximal stratum, 
the set of principal orbits $X_{reg}$,
which comes equipped with a natural quotient 
Riemannian metric which makes the projection $S(V)_{reg}\to X_{reg}$
into a Riemannian submersion. Moreover, $X$ is the completion of the 
convex open submanifold $X_{reg}$, hence Toponogov's globalization 
theorem~\cite{Petglob} says that $\kappa_X$ is the largest number $\kappa$
such that $X$ is an Alexandrov space of curvature bounded below by $\kappa$,
where $\kappa_X =\kappa_{\rho}$ is defined as in the introduction.

\subsection{Types of equivalence between representations}\label{types}

We say that two representations $\rho:G\to\OG V$ and $\tau:H\to\OG W$ 
are \emph{quotient-equivalent} if they have isometric orbit spaces~\cite{GL}. 
If, in addition, $\dim G<\dim H$, we say that $\rho$ is a \emph{reduction}
of~$\tau$. A representation that admits no reductions is called~\emph{reduced}. 

A special case of quotient-equivalence occurs when there is an isometry 
from $V$ to $W$ mapping $G$-orbits onto $H$-orbits. In this case
we say that $\rho$ and $\tau$ are \emph{orbit-equivalent}. 

\subsection{Local convexity and folding map}\label{folding}
For an isometric action of $G$ on $S(V)$ with orbit space $X$,
we denote the stratum of $X$ corresponding to an isotropy group $K$ by
$X_{(K)}$. Every stratum $X_{(K)}$ of $X$ is a (possibly incomplete, disconnected)
totally geodesic Riemannian submanifold of $X$ which is moreover a locally
convex subset. It follows that the infimum of
the sectional curvatures in $X_{(K)}$ is also bounded below by $\kappa_X$.

The set of fixed points of the isotropy group $K$ of $G$
is a subspace $W$ on which the normalizer $N_G(K)$ acts isometrically.  
Let $H:=N_G(K)/K$. Then $H$ acts on $W$ with trivial principal isotropy groups. 
The quotient $Y=S(W)/H$ admits a canonical map
$I_{(K)}:Y\to X$ which is $1$-Lipschitz, finite-to-one and length-preserving,
and an injective local isometry from an open dense subset of $Y$ onto 
$X_{(K)}$~\cite{GL3}. 
We will call $I_{(K)}$ 
the \emph{folding map} associated with $X_{(K)}$. 
If $\dim X_{(K)}\geq2$, we deduce that $\kappa_X\leq\kappa_Y$. 
We have proved (cf.~\cite[\S5]{GL3}):
\begin{prop}\label{local-convex}
Let $\rho:G\to\OG V$ be a representation and $X=S(V)/G$. 
Assume there is a non-principal stratum $X_{(K)}$ of dimension $d\geq2$
of $X$. Then there is another representation $\tau:H\to\OG W$ 
such that $Y=S(W)/H$ has dimension $d$ and $\kappa_X\leq\kappa_Y$.
\end{prop}

It is known that the folding map associated with the principal 
stratum $X_{reg}$ is a global isometry. The corresponding representation 
of $H$ on $W$ is called the \emph{principal reduction} 
of the representation $\rho$~\cite{str}.

\subsection{Rank and strata} \label{rankstr}
We quote from~\cite{GL4}:

\begin{lem} \label{minorbit}
Let a compact Lie group $G$ of dimension~$g$ and rank~$k$ 
act by isometries on $S(V)$.
Then:
\begin{enumerate}
\item[(i)] The smallest dimension 
of a $G$-orbit is at most $g-k+1$.
\item[(ii)] If the action has trivial principal isotropy groups, then
$X=S(V)/G$ contains a non-maximal stratum of dimension at least $k-2$.
\end{enumerate}
\end{lem}

\subsection{Index  estimates}
The following  result was proved in~\cite{GL4} and gives 
slightly more than can be directly 
obtained from O'Neill's formula. It already shows that ``most'' 
representations are not highly curved. 

\begin{lem} \label{mainlem}
Let a compact Lie group $G$ of dimension~$g$ and rank~$k$ 
act by isometries on $S^n$. Let $\ell$ denote the smallest
dimension of an orbit,
and let $m\geq 2$ denote the dimension of the orbit space $X=S^n/G$.
If $\kappa >1$  then $\ell\geq m-1$; in particular, $2g +2 -k \geq n$.
\end{lem}


A very similar reasoning yields the following 
improved index inequality:

\begin{lem}\label{improvedlem}
  Let $\rho:G\to\OG{n+1}$ be a highly curved 
representation with trivial principal
  isotropy groups.
Let $m$ be the dimension of $X$ and $g=n-m$ be the dimension of $G$.
Assume there exists a regular horizontal geodesic $\gamma $ in $S^n$, of length less 
than~$\pi$, intersecting singular orbits of dimensions~$\ell_1$, 
$\ell_2,\ldots,\ell_s$.
Then $g \geq (m-1)+ (g-\ell_1)+ (g-\ell_2)+\ldots+ (g-\ell_s)$.
In particular, if $s>1$ then $\ell_1+\ell_2 \geq n-1$.
\end{lem}

\subsection{Enlarging group actions}
We consider the situation in which the $G$-action on $S(V)$ 
is the restriction of the action of a compact Lie group $H$
that contains $G$ as a closed subgroup. We will
need the following extension of the results in~\cite[\S2.3]{GL4}.
Recall that polar representations~\cite{D} are exactly those whose
orbit space has constant curvature~$1$~\cite[Introd.]{GL2}. 

\begin{prop}\label{enlargement}
\begin{enumerate}
\item[(a)] Suppose an orthogonal representation $\rho:G\to\OG{V}$
is the restriction of another representation $\tau:H\to\OG{V}$, where $G$ is a closed
subgroup of $H$.  If the cohomogeneity of~$\tau$ is at least $3$,
then $\kappa_\rho\leq\kappa_\tau$.
\item[(b)] The following classes of representations
$\rho:G\to\OG{V}$ are not highly curved:
\begin{enumerate}
\item[(i)] Representations~$\rho$ as in~(a), where $\tau$ is polar
and has cohomogeneity at least~$3$.
\item[(ii)] Tensor products, where $G=G_1\times G_2$, $V=V_1\otimes_{\mathbb F}V_2$
and $\dim_{\mathbb F}V_i\geq3$ for $i=1$, $2$ ($\mathbb F=\R$, $\C$, $\Q$).
\item[(iii)] Direct sums $V=V_1\oplus V_2$, where the $G$-action on $V_2$
has cohomogeneity at least $2$.
\item[(iv)] Direct sums $V=V_1\oplus\cdots\oplus V_n$, where $n>2$.
\end{enumerate}
\end{enumerate}
\end{prop}

\Pf See~\cite{GL4} for (a), (i) and (ii). 
For~(iii), we take
$H=\SO{V_1}\times G_2$,
where $G_2=G/\ker\rho_2$, and note that $S(V)/H$ is the
suspension (or spherical cone) over $S(V_2)/G_2$. Since $\dim S(V_2)/G_2\geq1$,
the suspension over a non-constant geodesic in $S(V_2)/G_2$ is a
convex totally geodesic surface in $S(V)/H$ which is locally
isometric to the unit sphere~\cite[\S3.6.3]{BBI}, hence $\kappa_\tau=1$
and we can apply~(a). Finally, the case~(iv) is reduced to the previous
case simply by writing $V=V_1\oplus(V_2\oplus\cdots\oplus V_n)$
and noting that the cohomogeneity of $G$ on
$V_2\oplus\cdots\oplus V_n$ is bigger than one. \EPf


\section{Some interesting examples}

In this section we show that a few specific representations
are (are not) highly curved. These results are part of the proof
of Theorem~\ref{main}. 

\subsection{The curvature of complex weighted projective lines}\label{cwpl}

Let $\U1$ act on $\C\oplus\C$ with parameters $(a,b)$, namely,
$\xi\cdot(z,w)=(\xi^az,\xi^b,w)$ for $\xi\in\U1\subset\C$ and 
$z$, $w\in\C$. We assume 
that $a$ and $b$ are co-prime, positive integers, and $a\geq b$. The map
\[ F:(0,\pi/2)\times(0,2\pi)\to S^3,\quad F(r,\theta)=(\cos r,
e^{i\theta}\sin r) \]
meets all principal orbits, so the orbital metric in the principal stratum
of $X=S^3/\U1$ can be easily computed in terms of~$F$ to give
\[ g= dr^2 + \frac 14 \frac{a^2 \sin^2 2r}{a^2 \cos^2 r + b^2 \sin^2 r}\, d\theta^2. \]
This is a rotationally symmetric metric, whose Gaussian curvature is given
by 
\[ K(r) = \frac{3a^4 + 26a^2 b^2+ 3b^4+4(a^4-b^4)\cos 2r +(a^2-b^2)^2\cos 4r}{2(a^2+b^2+(a^2-b^2)\cos 2r)^2}. \]
We have 
\[ K'(r)=\frac{48a^2b^2(a^2-b^2)\sin 2r}{(a^2+b^2+(a^2-b^2)\cos 2r)^3} > 0, \]
so that
\[ K_{inf}=K(0+)=1+ 3 \frac{b^2}{a^2},\qquad K_{sup}=K(\frac\pi2-)=1+ 3 \frac{a^2}{b^2} \]
and hence $1<\kappa_X<4$, unless $a=b=1$ in which case $X$ is a $2$-sphere 
of constant curvature $4$. In any case, $X$ is highly curved. 

\subsection{The representation $(\SU2,\C^2\oplus\R^3)$}\label{c2r3}

In terms of quaternions, this representation is $(\SP1,\Q\oplus\Im\Q)$
where $q\cdot(x,y)=(qx,qyq^{-1})$. The only non-principal orbit in $S^6(1)$ 
corresponds to $x=0$. The map 
\[ \Q\oplus\Im\Q \to \Im\Q \cong\R^3, \qquad (x,y)\mapsto x^{-1}yx \]
is well-defined and constant on principal orbits, applies  
the regular part of $S^6(1)$ onto the interior of the closed ball $\bar B^3(1)$, and 
a neighborhood of the singular orbit to a neighborhood of the boundary
$\partial\bar B^3(1)=S^2(1)$. 
It follows that $X$ is topologically a $3$-sphere. A section of the 
above map over the regular set is
\[ B^3(1) \subset\R^3 \to S^6(1) \subset \Q \oplus\Im \Q,\qquad
v\mapsto (\sqrt{1-||v||^2},v) \]
which, in spherical coordinates, is written
\begin{eqnarray*}
 \lefteqn{(r,\theta,\varphi)\in(0,\pi/2)\times(0,2\pi)\times(0,\pi)\mapsto} \\
&&(\cos r,0,0,0,\sin r\cos\theta\sin\varphi,\sin r\sin\theta\sin\varphi,\sin r \cos\varphi)\in S^6(1)\subset\R^7. 
\end{eqnarray*}
One easily computes the inner products of the horizontal components of 
$\frac{\partial}{\partial r}$, $\frac{\partial}{\partial\theta}$, 
$\frac{\partial}{\partial\varphi}$ to obtain the orbital metric coefficients:
\[ g = dr^2 + \frac{\sin^2 2r}{10-6\cos 2r}(d\varphi^2 + \sin^2 \varphi\, d\theta^2). \]
This is a warped product $X_{reg}=[0,\pi/2)\times_f S^2(1)$ where 
$f(r) = \frac12 \frac{\sin 2r}{\sqrt{\cos^2 r + 4 \sin^2r}}$ is the coefficient 
of the metric associated to the complex weighted projective line of weights~$(1,2)$.  
The sectional curvatures of such a warped product 
are well known~\cite[\S~3.2.3]{petersen}, namely,
they all lie between $-f''/f$ and $(1-f')f^{-2}$. We have $\Ima(-f''/f)=(7/4,13)$ and
$\Ima((1-f')f^{-2})=[9,+\infty)$, so $\kappa_X=7/4$ and $X$ is 
highly curved. 

\subsection{The representation $(\SO3,\R^7)$}\label{so3}

This representation is induced from the real form $V$ of $(\SU2,\mathrm{Sym^6(\C^2)})$
given by 
\[ \mathrm{span}_{\mathbb R}\{e_1^6+e_2^6,i(e_1^6-e_2^6), e_1^5e_2-e_1e_2^5, i(e_1^5e_2+e_1e_2^5),
e_1^4e_2^2+ e_1^2e_2^4, i(e_1^4e_2^2-e_1^2e_2^4),ie_1^3e_2^3\}. \]
There is exactly one singular orbit, namely, that through $p=ie_1^3e_2^3$,
whose isotropy group is the (diagonal) maximal torus (circle). It is 
easy to find 
$g=\left(\begin{array}{cc}\alpha&-\bar\beta\\\beta&\bar\alpha\end{array}\right)\in\SU2$ 
such that $q=gp\neq-p$ is orthogonal to 
$T_p(Gp)=\mathrm{span}_{\mathbb R}\{e_1^4e_2^2+ e_1^2e_2^4, i(e_1^4e_2^2-e_1^2e_2^4)\}$, e.g.
any $\alpha$, $\beta$ with $|\alpha|^2=\frac12(1\pm\frac1{\sqrt5})$, 
$|\beta|^2=1-|\alpha|^2$ will do. It follows that there is a regular horizontal 
geodesic of length smaller than $\pi$ that meets $Gp$ in two points, namely, a 
minimizing geodesic segment between $p$ and $q$.
Since $\ell_1=\ell_2=2$ and $n=6$, Lemma~\ref{improvedlem} implies that 
this representation is not highly curved.

\subsection{The representation $(\SP1\times\SP1,\Q^3\otimes_{\mathbb H}\Q)$}\label{h3xh}

We will show that this representation is not highly curved. 
We consider a double quotient
\begin{diagram}
&&S^{11}&&\\
&\ldTo&&\rdTo&\\
Y=S^{11}/H&&&&Z=S^{11}/K=\Q P^2\\
&\rdTo&&\ldTo&\\
&&X=S^{11}/G&&\\
\end{diagram}
where: $H$ is $\SP1$ acting on the left, which we view as the 
representation of quaternionic type
$(\SU2,\mathrm{Sym}^5(\C^{2}))$; $K$ is $\SP1$ acting on the right; and $G=H\times K$. 

Note that $Y$ is a Riemannian orbifold (a quaternionic weighted projective
space). View the representation space of $H$ as
\[ \mathrm{span}_{\mathbb C}\{e_1^5,e_1^4e_2,e_1^3e_2^2,e_1^2e_2^3,e_1e_2^4,e_2^5\} \]
and take $p=e_1^5$. The isotropy group $H_p$ is the cyclic group
$\mathbb Z_5$, say with a generator $h=\mathrm{diag}(e^{i\omega},e^{-i\omega})$, 
where $\omega=2\pi/5$, the tangent space
\[ \Lh\cdot p =\mathrm{span}_{\mathbb R}\{ie_1^5,e_1^4e_2,ie_1^4e_2\} \]
and the normal space 
\[ N_p(Hp)=\mathrm{span}_{\mathbb C}\{e_1^3e_2^2,e_1^2e_2^3,e_1e_2^4,e_2^5\}.\]
It follows that $h$ acts on $\Lh\cdot p$ 
is $\mathrm{id}_{\mathbb R}\oplus R_{3\omega}$, 
where $R_{\theta}$ denotes a rotation of angle $\theta$ 
on an oriented $2$-plane. Similarly, the action of $h$ on $N_p(Hp)$ 
is $\mathrm{id}_{\mathbb R^2}\oplus R_{\omega}
\oplus R_{2\omega}\oplus R_{4\omega}$. 
Since the O'Neill tensor $A^H_p:\Lambda^2N_p(Hp)\to\Lh\cdot p$ is 
$H_p$-equivariant and  $4\pm2\neq0$, $3$ mod $5$, we find that
$w_1=e_1^2e_2^3$ and $w_2=e_1e_2^4$ satisfy $A^H_p(w_1\wedge w_2)=0$, that is, the
$2$-plane $\sigma$ spanned by $w_1$, $w_2$ projects to a $2$-plane of sectional
curvature $1$ in $Y$. 

The quaternionic structure on $\mathrm{Sym}^5(\C^2)$ is induced from that 
of $\C^2$. Since the latter maps $e_1$ to $e_2$ and $e_2$ to $-e_1$, the 
former maps $w_1=e_1^2e_2^3$ to $-e_1^3e_2^2$ and $w_2=e_1e_2^4$ to $e_1^4e_2$, 
so $\sigma$ is a totally real plane and maps to 
a $2$-plane of sectional curvature $1$ in $Z$. Equivalently, the 
O'Neill tensor of $S^{11}\to\Q P^2$ vanishes on $w_1\wedge w_2$. 

Let $x$ be the projection of $p$ to $X$. This is an isolated
singular point in $X$, but $p$ is an exceptional point of the $H$-action and a 
regular point of the $K$-action, so we have continuity at $p$ of the 
O'Neill tensors of Riemannian submersions to $Y$ and $Z$.
It follows that there is a sequence of $G$-regular points $p_n\to p$ 
and $2$-planes $\sigma_n$ tangent to $S^{11}$ at $p_n$ projecting to 
$2$-planes in $X$ with sectional curvature $\to1$. Hence $\kappa_X=1$. 

\subsection{The representation $(\SU2,\Q^2)$}\label{su2}

We will prove that this representation is not highly curved.
Let $G=\SU2$. We view this representation as the cubic symmetric power 
$V=\mathrm{Sym}^3(\C^{2*})$. Namely, write an arbitrary element $g\in\SU2$ as
\begin{equation}\label{g}
 g=\left(\begin{array}{cc}\alpha&-\bar\beta\\
                            \beta & \bar\alpha
\end{array}\right) 
\end{equation}
where $\alpha$, $\beta\in\C$, $|\alpha|^2+|\beta|^2=1$. 
This exhibits the matrix representation of the standard action  
of $\SU2$ on $\C^2$ with respect to the canonical basis $\{e_1,e_2\}$. 
Let $\{u,v\}$ be the dual basis of $\C^{2*}$. The action 
of $g$ on $\C^{2*}$ is represented by the matrix complex-conjugate 
to~(\ref{g}) with respect to this basis. 
Now an orthonormal basis of $V$ is given by 
\[ \left\{ \frac{u^3}{\sqrt6}, \frac{uv^2}{\sqrt2}, 
\frac{v^3}{\sqrt6},  \frac{u^2v}{\sqrt2}\right\}. \]
(We have chosen the order in the basis in view of the quaternionic structure
below.) In this basis, the action of $g$ on $V$ is represented 
by the matrix
\begin{equation}\label{g2} 
\left(\begin{array}{cccc}
\bar\alpha^3 & \sqrt3\bar\alpha\beta^2 & -\beta^3 & -\sqrt3\bar\alpha^2\beta \\
\sqrt3\bar\alpha\bar\beta^2 & \alpha(|\alpha|^2-2|\beta|^2) &
-\sqrt3\alpha^2\beta &\bar \beta(2|\alpha|^2-|\beta|^2) \\
\bar\beta^3 & \sqrt 3\alpha^2 \bar\beta & \alpha^3 & \sqrt3\alpha\bar\beta^2\\
\sqrt3\bar\alpha^2\bar\beta & \beta(|\beta|^2-2|\alpha|^2) & 
\sqrt3 \alpha\beta^2 & \bar\alpha(|\alpha|^2-2|\beta|^2) \end{array}\right) 
\end{equation}
 
On the level of Lie algebras, 
consider the basis
\begin{equation}\label{basis}
 \bf i=\left(\begin{array}{cc}i&0\\
                              0&-i
\end{array}\right),\quad 
j=\left(\begin{array}{cc}0&-1\\
                              1&0
\end{array}\right),\quad 
k=\left(\begin{array}{cc}0&-i\\
                           -i&0
\end{array}\right)
\end{equation}
of $\mathfrak{su}(2)$. 
We have $[{\bf i},{\bf j}]=2{\bf k}$ and cyclic permutations.
These matrices, viewed as elements of $\Lg$, operate on $V$ as
\[ i_L=\left(\begin{array}{cccc}-3i&&&\\
                              &i&&\\
                              &&3i&\\
                              &&&-i
\end{array}\right),\quad 
j_L=\left(\begin{array}{cccc}&&0&-\sqrt3\\
                         &&-\sqrt3&2\\
                         0&\sqrt3&&\\
                         \sqrt3&-2&&
\end{array}\right),\]
and
\[ k_L=\left(\begin{array}{cccc}&&0&i\sqrt3\\
                         &&i\sqrt3&2i\\
                         0&i\sqrt3&&\\
                         i\sqrt3&2i&&
\end{array}\right).
\]
Beware that $i_Lj_L\neq k_L$, etc.

\subsubsection{Quaternionic structure}

The matrix~(\ref{g2}) has the form
\[ \left(\begin{array}{cc}A&-\bar B\\B& \bar A\end{array}\right). \]
If we identify $V\cong\C^4$ with $\Q^2$ via the map
\[ x_1\frac{u^3}{\sqrt6}+ x_2\frac{uv^2}{\sqrt2}+ y_1
\frac{v^3}{\sqrt6}+ y_2 \frac{u^2v}{\sqrt2} \mapsto (x_1+jy_1,x_2+jy_2) \]
then the representation is given by left multiplication by the quaternionic
matrix $A+jB$. In particular, the normalizer of the representation of 
$G$ on $V$ contains another copy $G'$ of $\SU2$ acting on the right.
The action of $q\in\SP1\cong G'$ on $V$ is given by 
right multiplication of $\Q^2$ by $q^{-1}$.  
It is \emph{not} complex linear.

We describe the action of the basis elements~(\ref{basis}) 
of $\Lg'$ as follows:
\[ i_R(x_1+jy_1,x_2+jy_2)=((-ix_1)+j(-iy_1),(-ix_2)+j(-iy_2)), \]
\[ j_R(x_1+jy_1,x_2+jy_2)=(\bar y_1-j\bar x_1,\bar y_2-j\bar x_2), \] 
\[ k_R(x_1+jy_1,x_2+jy_2)=((-i\bar y_1)+j(i\bar x_1),(-i\bar y_2)+j(i\bar x_2)).
\] 

\subsubsection{Cohomogeneity one}

Since $G'$ normalizes $G$, it acts on $X=S(V)/G$. Indeed the group generated 
by $G$ and $G'$ is the full normalizer $K=\SO4$ of $G$ in $\OG V=\OG8$.
The representation of $K$ on $V$ is the isotropy representation of the 
symmetric space $\G/\SO4$, of rank $2$. It follows that $X/G'$ is 
one-dimensional and in fact isometric to the interval $[0,\pi/6]$. 
Now the action of $G'$ on $X$ has cohomogeneity one. 
A $K$-horizontal geodesic is given by
\[ \gamma(t) = \cos t \frac{u^3}{\sqrt6} + \sin t \frac{uv^2}{\sqrt2}. \]
It suffices to compute the sectional curvatures of $X$ along the projection 
of~$\gamma$.

\subsubsection{Natural frames}

For future reference, we compute:
\[ i_R\gamma(t) = -\cos t \frac{iu^3}{\sqrt6} - \sin t \frac{iuv^2}{\sqrt2} \]
\[ j_R\gamma(t) = -\cos t \frac{v^3}{\sqrt6} - \sin t \frac{u^2v}{\sqrt2} \]
\[ k_R\gamma(t) = \cos t \frac{iv^3}{\sqrt6} + \sin t \frac{iu^2v}{\sqrt2} \]
\[ i_L\gamma(t) = -3\cos t \frac{iu^3}{\sqrt6} + \sin t \frac{iuv^2}{\sqrt2} \]
\[ j_L\gamma(t) = \sqrt3 \sin t \frac{v^3}{\sqrt6}+(-2\sin t+\sqrt3\cos t)\frac{u^2v}{\sqrt2} \]
\[ k_L\gamma(t) = \sqrt3\sin t \frac{iv^3}{\sqrt6}
+(2\sin t + \sqrt 3 \cos t )\frac{iu^2v}{\sqrt2}. \]
It is useful to note that there is an orthogonal decomposition
\[ T_pS^7 = \langle\gamma'(t)\rangle\oplus
\langle i_Rp,i_Lp\rangle\oplus\langle j_Rp,j_Lp\rangle\oplus\langle k_Rp,k_Lp\rangle\]
where $p=\gamma(t)$. 

\subsubsection{The Weyl group}\label{weyl}

The singular points
of the $K$-action on $X$ are~$p_1=\gamma(0)$ and~$p_2=\gamma(\pi/6)$. 
Their isotropy groups are given by 
\[ K_{p_1}=\langle(e^{i\theta},e^{-3i\theta}), (j,j)\rangle, \]
\[ K_{p_2}=\langle(e^{j\theta},e^{j\theta}), (i,i)\rangle. \]
Note that 
\[ K_{princ}=\{\pm(1,1),\pm(i,i),\pm(j,j),\pm(k,k)\}. \]
The reflections at $p_1$, $p_2$ are given by 
\[ w_1=e^{i\pi/4}(1,-1),\quad w_2=e^{j\pi/4}(1,1). \]
Let $w=w_1w_2$. Then $w$ maps $\gamma(t)$ to $\gamma(t-\pi/3)$
and acts by conjugation on $K_{princ}$ by
cyclically permuting $(i,i)$, $(j,j)$, $(k,k)$. 

\subsubsection{The O'Neill tensor} 

The vertical space at $p=\gamma(t)\in S(V)$ 
is spanned by $i_Lp$, $j_Lp$, $k_Lp$; this is an orthogonal frame.  
Also $i_Rp$, $j_Rp$, $k_Rp$ is an orthogonal frame. The only nonzero 
inner products between vectors in the two sets are: 
\[ i_0(t):=\langle i_L\gamma(t),i_R\gamma(t)\rangle = 4\cos^2 t-1, \]
\[ j_0(t):=\langle j_L\gamma(t),j_R\gamma(t)\rangle=
\langle i_L\gamma(t+\pi/3),i_R\gamma(t+\pi/3)\rangle\]
and 
\[k_0(t):=\langle k_L\gamma(t),k_R\gamma(t)\rangle=
\langle i_L\gamma(t+2\pi/3),i_R\gamma(t+2\pi/3)\rangle,\]
using the action of the Weyl group element~$w$. Moreover
\[ ||i_Rp||^2=||j_Rp||^2=||k_Rp||^2=1, \]
\[ ||i_L\gamma(t)||^2= 1+8\cos^2t, \]
\[ ||j_L\gamma(t)||^2=||i_L\gamma(t+\pi/ 3)||^2\]
and
\[ ||k_L\gamma(t)||^2=||i_L\gamma(t+2\pi/ 3)||^2.\]

Denote by $i_R^h(t)$ the vector field along $\gamma$ in $S^7$ 
given by the horizontal projection of $i_R\gamma(t)$. 
Then 
\[ i_R^h(t)= i_R\gamma(t)- I_0(t)i_L\gamma(t) \]
where $I_0(t)=i_0(t)/||i_L\gamma(t)||^2$; put also
$I(t):=||i_R^h(t)||$. Define similarly $j_R^h$, $k_R^h$, 
$J_0$, $K_0$, $J$, $K$.  
A natural horizontal orthonormal frame along $\gamma$ 
is now given by $\gamma'=\partial/\partial t$, $i_R^h/I$, $j_R^h/J$, 
$k_R^h/K$. 

We use O'Neill's formula to show that there is a value of $t$ for which 
the plane spanned by $i_R^h$, $j_R^h$ projects
to a plane of curvature $1$ in $X$. In fact, equivariantly extend
$i_R^h$, $j_R^h$, $k_R^h$ to vector fields in $S^7$, denote the Levi-Civit\`a 
connection of $S^7$ by $\nabla$ and  
the O'Neill tensor of $S^7\to X$ by~$A$. Then
\begin{eqnarray*}
A_{i_R^h}j_R^h &=& (\nabla_{i_R^h}j_R^h)^v \\
&=& \langle\nabla_{i_R^h}j_R^h,k_L   \rangle \frac{k_L}{||k_L||^2},
\end{eqnarray*}
where 
\begin{eqnarray*}
\langle\nabla_{i_R^h}j_R^h,k_L   \rangle&=& -\langle j_R^h,\nabla_{i_R^h}k_L   \rangle \\
&=& -\langle j_R^h,k_L(i_R^h)   \rangle\quad \text{(since $k_L$ is a linear vector field)} \\
&=& k_0(t)+I_0(t)j_0(t)+J_0(t)i_0(t)-I_0(t)J_0(t)\langle k_Li_L\gamma(t),j_L\gamma(t)\rangle.
\end{eqnarray*}
The sectional curvature is
\begin{eqnarray*}
K_X(i_R^h\wedge j_R^h)|_t & = & 1 + 3 \frac{||(A_{i_R^h}j_R^h)_t||^2}{I(t)^2J(t)^2}\\ & = & 1 - 27 \frac{(-2 \sqrt3 - \sqrt3 \cos{2t} +\sin{2t} + 
      4 \sin{4t})^2}{P(t)} 
\end{eqnarray*}
where 
\begin{eqnarray*}
 \lefteqn{P(t)=(5 + 4 \cos{2t}) (5 - 2 \cos{2t} + 
      2 \sqrt3 \sin{2t})} \\
&& (-10 + 2\cos{2t} - 5\cos{4t} + 
      4 \cos{6 t} + 2 \sqrt3 \sin{2 t} + 5 \sqrt3 \sin{4t}).
\end{eqnarray*}
It is easy to prove that there is $t_0\in(0,\pi/6)$ such that 
\[ -2 \sqrt3 - \sqrt3 \cos{2t} +\sin{2t} + 
      4 \sin{4t} =0, \]
      which shows the existence of planes with curvature~$1$.
      Indeed $t_0=\frac{\pi}3-\frac12\arccos\frac14\approx 0.38814$. 

\begin{figure}
\includegraphics[scale=.5]{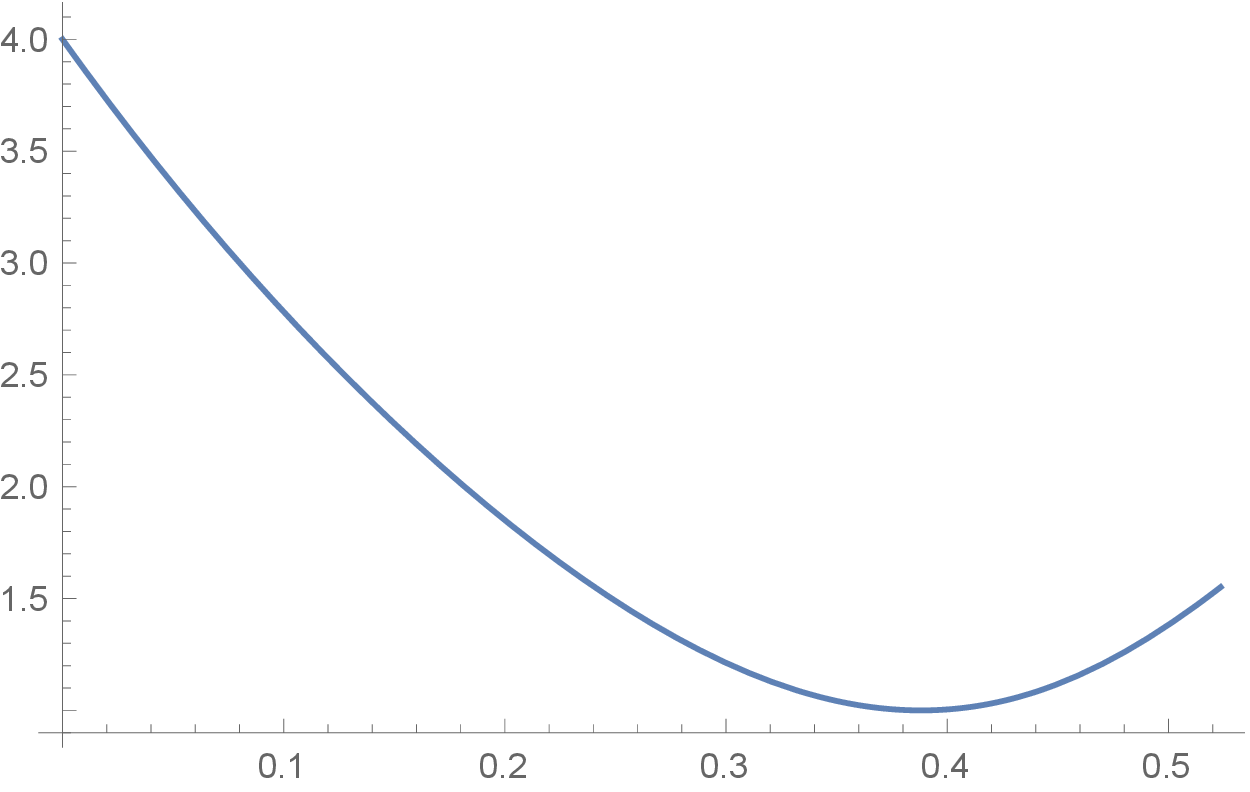}
\caption{Graph of $K(i_R^h\wedge j_R^h)$.}
\end{figure}

\subsection{The representation $(\U2,\C^4)$}\label{u2}

This representation is an enlargement of that in subsection~\ref{su2};
we retain the notation therein. View 
$H=\U2$ as the group generated by $G\cong\SU2$ and a circle 
subgroup of $G'\cong\SU2$, and denote the orbit space $S^7/H$ by $Y$. 
We will show that $Y$ is highly curved.

\subsubsection{Description of all $2$-planes in $X=S^7/\SU2$ with 
sectional curvature~$1$}\label{thorpe}

We use ideas of~\cite{thorpe} in dimension~$4$.
Let $x_t\in X$ be the projection of $\gamma(t)$.
The curvature operator $R_t:\Lambda^2T_{x_t}X \to \Lambda^2T_{x_t}X$
is self-adjoint, and its matrix with respect to the orthonormal basis 
\begin{equation}\label{frame}
  \frac{\partial}{\partial t}\wedge\frac{i_R^h}I, 
\frac{\partial}{\partial t}\wedge\frac{j_R^h}J, 
\frac{\partial}{\partial t}\wedge\frac{k_R^h}K, 
\frac{j_R^h}J\wedge\frac{k_R^h}K,\frac{k_R^h}K\wedge\frac{i_R^h}I,\frac{i_R^h}I\wedge\frac{ j_R^h}J 
\end{equation}
is
\[ \left(\begin{array}{cccccc}
a_1 &&& b_1 &&\\
& a_2 &&& b_2 & \\
&& a_3 &&& b_3 \\
b_1 &&& c_1 &&\\
& b_2 &&& c_2 & \\
&& b_3 &&& c_3 \\
\end{array}\right)=  \left(\begin{array}{cc}A & B \\ B& C\end{array}\right) \]
where $a_i$, $b_i$, $c_i$ are smooth functions 
of $t\in(0,\pi/6)$. 
The diagonal elements of $A$ and $C$ are sectional curvatures and,
by the Bianchi identity, the trace of $B$ is zero. Note that $c_3$ was computed in
the previous section, and the other functions are 
computed similarly; their explicit values are listed in the appendix.
In view of the acton of the Weyl group~(subsection~\ref{weyl}),
these functions satisfy
\[ a_2(t)=a_1\left(t+\frac{\pi}3\right),\ a_3(t)=a_1\left(t+\frac{2\pi}3\right), \]
\begin{equation}\label{w}
 b_2(t)=b_1\left(t+\frac{\pi}3\right),\ b_3(t)=b_1\left(t+\frac{2\pi}3\right), 
\end{equation}
\[ c_2(t)=c_1\left(t+\frac{\pi}3\right),\ c_3(t)=c_1\left(t+\frac{2\pi}3\right). \]

Let $Z(\tilde R_t)$, where $\tilde R_t=R_t-I$, denote the subset of the 
Grassmannian $G_t:=\mathrm{Gr}_2(T_{x_t}X)$ consisting of points where the 
sectional curvature is $1$. We will show that 
$\ker(\tilde R_t-\mu(t)\ast)\cap G_t$ is non-empty, where 
$\ast$ denotes the Hodge star operator on~$\Lambda^2T_{x_t}X$ 
and~$\mu(t)$ is a certain 
smooth function. It will follow
from the remark after Theorem~2.1 and Theorem~4.1 in~\cite{thorpe} that 
$Z(\tilde R_t)=\ker(\tilde R_t-\mu(t)\ast)\cap G_t$.  

Let $\mu_i^{\pm}(t,\epsilon)=b_i(t)\pm\sqrt{(a_i(t)-1)(c_i(t)-1)}$ for
$i=1$, $2$, $3$. Take
$\mu=\mu_1^{-}$ and put $\alpha_i(t)=a_i(t)-1$, $\beta_i(t)=b_i(t)-\mu(t)$,
$\gamma_i(t)=c_i(t)-1$.
Using the explicit formulae in the appendix, one checks that
\[ ((b_1-b_2)^2-\alpha_1\gamma_1-\alpha_2\gamma_2)^2=4\alpha_1\alpha_2\gamma_1\gamma_2
\]
and 
\[ (b_1-b_2)^2-\alpha_1\gamma_1-\alpha_2\gamma_2>0,\ b_1>b_2. \]
It follows that 
\[ \mu_1^-(t)=\mu_2^+(t) \qquad\text{for all $t$.} \]
Similarly, one checks that 
\[ \mu_1^-(t)=\left\{\begin{array}{ll}
    \mu_3^+(t) & \text{for $t\leq t_0$,} \\
    \mu_3^-(t) & \text{for $t\geq t_0$.} \end{array}\right. \]

It immediately follows that 
$\ker(\tilde R_t-\mu(t)\ast)$ is $3$-dimensional and spanned by
\[ -\beta_1(t)\frac{\partial}{\partial t}\wedge\frac{i_R^h}I+\alpha_1(t)\frac{j_R^h}J\wedge\frac{k_R^h}K,\quad
-\beta_2(t)\frac{\partial}{\partial t}\wedge\frac{j_R^h}J+\alpha_2(t)\frac{k_R^h}K\wedge\frac{i_R^h}I \]
and 
\[ -\beta_3(t)\frac{\partial}{\partial t}\wedge\frac{k_R^h}K+\alpha_3(t)\frac{i_R^h}I\wedge\frac{j_R^h}J. \]
Let $r_1$, $r_2$, $r_3$ be the corresponding coordinates
on $\ker(\tilde R_t-\mu(t)\ast)$. The Pl\"ucker and 
normalization relations defining 
$\ker(\tilde R_t-\mu(t)\ast)\cap G_t$ now are
\[ \sum_{i=1}^3 r_i^2\alpha_i\beta_i=0,\qquad \sum_{i=1}^3 r_i^2(\alpha_i^2+\beta_i^2)=1. \] 
Solving these relations yields:
\begin{equation}\label{r1r2}
 r_1=\left\{\begin{array}{ll}
\pm\frac{1}{\sqrt{A_1(t)}}\cosh \theta&\mbox{if $t<t_0$} \\
\theta & \mbox{if $t=t_0$} \\
\frac{1}{\sqrt{-A_1(t)}}\sinh \theta&\mbox{if $t>t_0$,} \end{array}\right. 
\quad 
r_2=\left\{\begin{array}{ll}
\frac{1}{\sqrt{-A_2(t)}}\sinh \theta&\mbox{if $t<t_0$} \\
\pm\left(-\frac{\alpha_1(t_0)\beta_1(t_0)}{\alpha_2(t_0)\beta_2(t_0)}\right)^{1/2}\theta & \mbox{if $t=t_0$} \\
\pm\frac{1}{\sqrt{A_2(t)}}\cosh \theta&\mbox{if $t>t_0$} \end{array}\right. 
\end{equation}
and
\[ r_3^2= \frac{1-r_1^2(\alpha_1^2(t)+\beta_1^2(t))-r_2^2(\alpha_2^2(t)+\beta_2^2(t))}{\alpha_3^2(t)+\beta_3^2(t)}, \]
where 
\[ A_1=\alpha_1^2+\beta_1^2
-\frac{\alpha_1\beta_1}{\alpha_3\beta_3}(\alpha_3^2+\beta_3^2),
\quad
A_2=\alpha_2^2+\beta_2^2
-\frac{\alpha_2\beta_2}{\alpha_3\beta_3}(\alpha_3^2+\beta_3^2). \]
We deduce that for each $t\in(0,\pi/6)$ there is a one-parameter family of 
$2$-planes in $T_{x_t}X$ with sectional curvature $1$. 

\subsubsection{O'Neill's formula and the sectional curvatures of $S^7/\U2$}

We will use an argument based on a double quotient similar 
to that in subsection~\ref{h3xh}:
\begin{diagram}
&&S^7&&\\
&\ldTo&&\rdTo&\\
X=S^7/\SU2&&&&S^7/\U1=\C P^3\\
&\rdTo&&\ldTo&\\
&&Y=S^7/\U2&&\\
\end{diagram}

For some nonzero $\xi\in\Lh/\Lg=\mathfrak{u}(2)/\mathfrak{su}(2)$, 
the induced $\xi_R$ is a unit vertical vector 
field of $X\to Y$. Let now $u$, $v$ be linearly independent
tangent vectors to $S^7$, horizontal 
with respect to $S^7\to Y$. Then 
\begin{equation}\label{xiruv}
 \langle \xi_R(u),v\rangle 
\end{equation}
is a component of the O'Neill tensor of $S^7\to Y$, evaluated at $u\wedge v$,
which is complementary to the O'Neill tensor of $S^7\to X$; however, note that 
$\xi_R$ is not orthogonal to the vertical distribution of $S^7\to X$ (spanned 
by $i_L$, $j_L$, $k_L$). The inner product~(\ref{xiruv}) 
is also a component of the 
O'Neill tensor of $S^7\to \C P^3$, where the circle action for this 
Hopf action is infinitesimally generated by $\xi_R$. Let $\sigma$ be the 
$2$-plane tangent to $X$ which is the projection of $u\wedge v$, and assume 
that $\sigma$ has sectional curvature $1$. Then $\sigma$ projects to a   
$2$-plane of curvature $1$ in $Y$ 
if and only if $u\wedge v$ projects to a totally real $2$-plane in $\C P^3$;
in fact, $\xi_R$ induces the complex structure of $\C P^3$.  

Our method to prove that $\kappa_Y>1$ is to show that no 2-plane in $X$ with 
sectional curvature $1$, horizontal with respect to $X\to Y$, can correspond to 
a totally real $2$-plane in $\C P^3$. We first show it suffices to consider
$2$-planes in $X$ with sectional curvature $1$ 
\emph{along the projection of the geodesic $\gamma$}, as long as we 
take into account 
also the non-horizontal planes with respect to $X\to Y$. In fact,   
let $\sigma$ be a $2$-plane in $T_xX$ with sectional curvature $1$ and 
horizontal with respect to $X\to Y$, where $x\in X$ projects
to a regular point of $Y$.
There is $g\in G'$ such that $gx=x_t$ 
for some $t\in(0,\pi/6)$. Now $g_*\sigma$ is a $2$-plane in $T_{x_t}X$ with 
sectional curvature $1$ and 
it is horizontal with respect to 
$g_*\xi_R=(\mathrm{Ad}_g\xi)_R$ which is in general different from $\xi_R$,
but 
\[ \langle \xi_R(u),v\rangle = \langle (\mathrm{Ad}_g\xi)_R(gu),gv\rangle; \]
note that in principle
$\mathrm{Ad}_g\xi$ can be parallel to any element of $\Lg'$.
Conversely, given a $2$-plane $\sigma$ in $T_{x_t}X$
with sectional curvature~$1$ for some $t\in(0,\pi/6)$, 
represented by $u\wedge v$ where $u$, $v$ are 
vectors tangent to $S^7$, horizontal with respect to $S^7\to X$,
we observe that $\frac{d}{dt}x_t$ does not belong to $\sigma$
(since the $\alpha_i$ are positive on $(0,\pi/6)$).
Therefore there is a unique, up to sign, unit vector field $n_R$ on $S^7$, 
which is normal to $u$, $v$, $\gamma'(t)$, where $n\in\Lg'$
(cf.~(\ref{n})). 
We choose $g\in G'$ such that 
$\mathrm{Ad}_gn=\xi$ so that $\xi_R(g\gamma(t))=g_*(n_R\gamma(t))$ is normal to 
$gu\wedge gv$, which represents the $2$-plane $g\sigma$ 
in $T_{gx_t}X$ with sectional curvature~$1$
and horizontal with respect to $X\to Y$. 

Next we apply the method. 
Let $\sigma$ be a $2$-plane in $T_{x_t}X$ with sectional curvature $1$. 
Let 
\[ u = u_0\frac{\partial }{\partial t}+ u_1\frac{i_R^h}I+ u_2\frac{j_R^h}J+ 
u_3\frac{k_R^h}K \]
and 
\[ v = v_0\frac{\partial }{\partial t}+ v_1\frac{i_R^h}I+ v_2\frac{j_R^h}J+ 
v_3\frac{k_R^h}K \]
be tangent vectors to $S^7$ such that $u\wedge v$ projects to $\sigma$;
let $(\sigma_{01},\sigma_{02},\sigma_{03},\sigma_{23},\sigma_{31},\sigma_{12})$
be the coordinates of $\sigma$ in the basis~(\ref{frame}), so that 
$\sigma_{01}=u_0v_1-u_1v_0$ etc.
The unit normal vector is induced by the following \emph{non-zero}
element of~$\Lg'$:
\begin{equation}\label{n}
 n = (JK\sigma_{23}\,\textbf i+KI\sigma_{31}\,\textbf j+ 
IJ\sigma_{12}\,\textbf k)/\sqrt{J^2K^2\sigma_{23}^2+K^2I^2\sigma_{31}^2+I^2J^2\sigma_{12}^2}. 
\end{equation}
We can now compute:
\begin{eqnarray}\label{quadratic}
\lefteqn{\langle n_R(u),v\rangle = r_1^2 \alpha_1(t)(\alpha_1(t)E(t)-\beta_1(t)F(t))}\\ \nonumber
\qquad&&+r_2^2 \alpha_2(t)(\alpha_2(t)E(t+\pi/3)-\beta_2(t)F(t+\pi/3))\\ 
\nonumber
\qquad&&+r_3^2 \alpha_3(t)(\alpha_3(t)E(t+2\pi/3)-\beta_3(t)F(t+2\pi/3)),
\end{eqnarray}
where 
\[ E(t) = 1-J_0(t)j_0(t)-K_0(t)k_0(t)-J_0(t)K_0(t)i_0(t) \]
and
\[ F(t)= 2 J(t)K(t)I_0(t)\sin{2t}/I(t). \]
Explicit formulae for $E$ and $F$ are given in the appendix. 

Of course we have $\alpha_2(t)=\alpha_1(t+\frac{\pi}3)$,
$\alpha_3(t)=\alpha_1(t+\frac{2\pi}3)$, $\gamma_2(t)=\gamma_1(t+\frac{\pi}3)$
and $\gamma_3(t)=\gamma_1(t+\frac{2\pi}3)$ for $0<t<\pi/6$ (cf.~(\ref{w})), 
but the situation for the 
$\beta_i$'s is more complicated since it involves $\mu$. To remedy 
this situation, we introduce $\tilde\beta_1$. 
From the appendix we read  
\[ \beta_1(t)=27\frac{|1-4\cos 2t|}{\sqrt{(5+4\cos2t)^3(21-20\cos2t+8\cos4t)}}. \]
Put 
\[ \tilde\beta_1(t)=-27\frac{1-4\cos 2t}{\sqrt{(5+4\cos2t)^3(21-20\cos2t+8\cos4t)}}. \]
Then 
\begin{eqnarray*}
 \tilde\beta_1(t)&=&\beta_1(t)\quad\text{for $0<t<\pi/6$;}\\
 \tilde\beta_1(t+\frac{\pi}3)&=&-\beta_1(t+\frac{\pi}3)=\beta_2(t)\quad\text{for $0<t<\pi/6$;}\\
 \tilde\beta_1(t+\frac{2\pi}3)&=&\left\{\begin{array}
{ll}-\beta_1(t+\frac{2\pi}3)=\beta_3(t)\quad\text{for $0<t\leq t_0$,}\\
\beta_1(t+\frac{2\pi}3)=\beta_3(t)\quad\text{for $t_0<t\leq\pi/6$.}
\end{array}\right.
\end{eqnarray*}
Now we can rewrite~(\ref{quadratic}) as
\begin{equation}\label{quadratic2}
 \langle n_R(u),v\rangle = r_1^2 C_1(t) + r_2^2C_2(t) + r_3^2 C_3(t) 
\end{equation}
where                         
\[ C_1(t)=\alpha_1(t)\left[\alpha_1(t)E(t)-\tilde\beta_1(t)F(t)\right] \]
and $C_2(t)=C_1(t+\frac{\pi}3)$, $C_3(t)=C_1(t+\frac{2\pi}3)$
for $0<t<\pi/6$. Finally 
\[ \alpha_1(t)>0\quad\text{for all $t\in\R$,} \]
and we compute
\begin{equation}\label{54}
 \alpha_1(t)E(t)-\tilde\beta_1(t)F(t)=\frac{-54(1+2\cos4t)^2}{(5+4\cos2t)^2(21-20\cos2t+8\cos4t)}\leq0 
\end{equation}
for all $t\in\R$, and on the interval $[0,5\pi/6]$ 
it vanishes precisely 
for~$t=\pi/6$, $\pi/3$, $2\pi/3$ and~$5\pi/6$.

It is clear that the quadratic 
form~(\ref{quadratic2}) is nonpositive everywhere and  
could only non-trivially vanish at the endpoints $t=0$ and $t=\pi/6$. 
On the other hand, $C_1$ is bounded away
from zero near $t=0$ and $C_2$ is bounded away from zero near $t=\pi/6$
(cf.~(\ref{54}) and~Fig.~\ref{2}). 
Moreover~(\ref{r1r2}) shows that $r_1\to\pm3/4$ as $t\to0$ and
$r_2\to\pm7/324$ as $t\to\pi/6$. This proves that 
$\langle n_R(u),v\rangle$ is bounded away from zero on 
the set of $2$-planes of sectional curvature $1$ of $X$
and finishes the proof that $\kappa_Y>1$.

\section{Main result}
We are going to finish the proof of 
Theorem \ref{main} in this section. Let $\rho:G\to\OG{V}$ be a representation of a compact Lie group and assume that $X=S(V)/G$ has dimension~$m\geq2$.

\subsection{Polar case}\label{polar}
This is precisely the case in which the orbit space $X=S(V)/G$ is a good 
Riemannian orbifold of constant curvature~$1$~\cite[Introd.]{GL2}. In case 
of connected 
groups, these representations are classified and are orbit-equivalent to 
isotropy representations of symmetric spaces~\cite{D}. 

\subsection{Disconnected case}
Let $\rho _0$ be the restriction of $\rho$ to the identity component 
$G^0$ of $G$.
Then the projection $X_0=S(V) /G^0 \to X=S(V)/G$ is a Riemannian covering 
over the set of regular points of $X$. We deduce 
$\kappa _{\rho} =\kappa _{\rho_0}$.  
This shows it suffices to prove the results for representations of 
connected compact Lie groups $G=G^0$.

\subsection{Reducible  case}\label{red}
Assume that the representation $\rho:G\to\OG V$ is reducible. 
We will prove that $\rho$ is as in cases~(ii), (iii), (iv), (v) 
of Theorem~\ref{main}. Note that~$X$ 
has diameter at least~$\pi/2$ (cf.~\cite{GL}). 
This already implies  $\kappa_X \leq 4$. 
In view
of Proposition~\ref{enlargement}, we further
know that  $\kappa_X=1$ unless $\rho$ is the sum of two 
(irreducible) representations of cohomogeneity one.
In the latter case, assume $G$ is connected and $\rho$ is non-polar. 
Then either $\rho$ is listed in Tables 2 and~3 in~\cite[section 6]{GL3}, or
$X$ is not a Riemannian orbifold (i.e.~$\rho$ is not 
\emph{infinitesimally polar}). The representations in Table~2
have good Riemannian orbifolds of constant curvature~$4$ as orbit spaces
(case~(ii) of Theorem~\ref{main}),
and among those listed in Table~3, the only reduced representation is case~11,
which yields a complex weighted projective line as 
orbit space (case~(iii) of Theorem~\ref{main},
discussed in subsection~\ref{cwpl}), 
and the other representations reduce to a $\mathbb Z_2$-extension
of case~11. 
Going through the proof of Proposition~2 in~\cite{GL3}, we see
that in the non-infinitesimally polar case, $\rho$ must be as
in cases~(iv) or~(v) of Theorem~\ref{main}. 
Case~(iv) is
analyzed in subsection~\ref{c2r3}. In case~(v) the group
lies in between $\SP m$ and $\SP m\sf T^2$, so
Proposition~\ref{enlargement}(a) and~\cite[Table~1]{GL3} yield $\kappa=4$. 

\medskip

We hereafter assume~$\rho$ is irreducible. 

\subsection{The case $\mathrm{rank}(G)=1$}
Every representation of $\U1$ with cohomogeneity
at least~$3$ is reducible, so  
we may assume $G$ is covered by $\SU 2$. 
According to Lemma~\ref{mainlem},
the only irreducible, non-polar representations that need to be considered
are $(\SO3,\R^7)$ and  $(\SU2,\Q^2)$. They were examined
respectively in subsections~\ref{so3} and~\ref{su2}, and are not
highly curved.

\subsection{Initial cases}\label{initial}
We begin with $m=2$.
Due to \cite{str}, in this case 
the classification of irreducible representations of connected 
groups with cohomogeneity~$3$ yields good Riemannian orbifolds of constant 
curvature~$4$~\cite[Table~II, case~III]{str}. 
Assume now that $m=3$.  
By the classification of irreducible representations of connected groups 
of cohomogeneity~$4$~\cite[Theorem~1.11]{GL}, 
our representation either has a toric reduction 
(i.e.~reduces to a finite extension of a 
torus action) and hence 
is not highly curved, or is given by the action of $\U 2$ on $\C^4$, 
which has already been
discussed in subsection~\ref{u2} and is case~(i) of Theorem~\ref{main}.
Consider now~$m=4$. 
By the classification of irreducible representations of connected groups 
of cohomogeneity~$5$~\cite[Theorem~1.11]{GL}, 
our representation either has a toric reduction or reduces to the action
of $\SO3\times \U 2$ on $\R ^{12} =\R^3\otimes _{\R} \R^4$.
This representation is the restriction of a polar representation
of~$\SO3\times\SO4$, so it is not highly curved
by Proposition~\ref{enlargement}.

\subsection{Formulation}
We are going to complete the proof 
of Theorem \ref{main} by proving that there exist no 
irreducible representations $\rho$ of  connected compact Lie groups 
with $m\geq5$ that are highly curved. Suppose, to the contrary,
that there exists such a representation $\rho:G\to\OG V$.
We may assume that $m$ is minimal among all such examples. We may also
assume that for this $m$, the number $g=\dim G$ is minimal among all
such examples. We fix $\rho$ throughout the proof.

\subsection{Reduction}
By the assumption on the minimality of $g$, the representation of
$\rho$ is \emph{reduced} (cf.~subsection~\ref{types}): for any other
representation $\tau:H\to\OG W$ such that $S(W)/H$ is isometric to $X$,
we have $\dim (H) \geq g$. In particular, this implies that the action of
$G$ on $S(V)$ has trivial principal isotropy groups.

\subsection{Type of representation}\label{type}

We claim that the normalizer $N$ of $\rho(G)$ in~$\OG V$ has $\rho(G)$
as its identity component. Otherwise, we find a connected subgroup $H$ 
of $\OG V$
containing $\rho(G)$ with one dimension more. The inclusion $\tau:H\to \OG V$
is an irreducible representation and an enlargement of $\rho$. 
The quotient space $Y=S(V)/H$ has dimension at least $m-1\geq4$, and
$\kappa_\tau\geq\kappa_\rho>1$ due to Proposition~\ref{enlargement}.
Note that $\tau$ and $\rho$ cannot have the same orbits, for this
would contradict the triviality of principal isotropy groups of~$\rho$. 
In view of the minimality of~$m$, this implies that $m=5$ and $\dim Y=4$, but this
is in contradiction with subsection~\ref{initial}.

We deduce that $\rho$ cannot be of quaternionic type, and in case it is
of complex type, $G$ is covered by $\U1\times G'$ for a connected compact
Lie group~$G'$. 

\subsection{Consequences}
We already know that $\rho$ is not polar and $\mathrm{rank}(G)\geq2$. 
Proposition~\ref{local-convex} and Lemma~\ref{minorbit}
together with the choice of~$m$ imply that: 
either $X$ does not contain strata of dimensions between
$2$ and $m-1$, whence the rank of $G$ is at most~$3$; or 
$X$ does contain a stratum $X_{(K)}$ of dimension in that range
whose associated folding map $I_{(K)}$ is defined on the orbit space of a 
representation falling into one of the cases~(i), (ii),
(iii), (iv) or (v), so
that $\dim Y\leq5$ and thus 
the rank of $G$ is at most~$7$. 

In view of subsection~\ref{type} and Proposition~\ref{enlargement},
there only remain the following possibilities:
\begin{enumerate}
\item $\rho$ is an irreducible representation of real type
  a simple Lie group $G$;
\item $G=\U1\times G'$ and $\rho=\theta\otimes\rho'$, where $\theta$ is the
  representation of $\U1$ on $\C$ and $\rho'$ is an irreducible representation of complex  type of a simple Lie group $G'$;
\item $G=\SP1\times G'$ and $\rho=\phi\otimes\rho'$, where $\phi$ is the
  representation of $\SP1$ on~$\Q$ and $\rho'$ is an irreducible representation of
  quaternionic type of a simple Lie group $G'$;
\item  $G=G_1\times G_2$ and $\rho=\rho_1\otimes\rho_2$,
where $V=\mathbb F^2\otimes_{\mathbb F} V_2$, $\dim_{\mathbb F}V_2\geq2$ and
$\mathbb F$ can be~$\R$, $\C$ or~$\Q$.
\end{enumerate}




\subsection{Kollross' tables and first three cases}\label{simple}

\subsubsection{Case~(a)}
In view of Lemmma~\ref{mainlem},

\begin{equation}\label{real-bd1}
  \dim V\leq 2\dim G + 3 - \mathrm{rk}\,G
  \end{equation}
and
\begin{equation}\label{real-bd2}
  \dim V\leq 2\dim G+2.
\end{equation}
Using~(\ref{real-bd2}) and~\cite[Lemma~2.6]{K},
we deduce that $\rho$ must be one of:
$(\SO7,\Lambda^3\R^7)$,
$(\Spin{15},\R^{128})$ (half-spin),
$(\SO8,\Lambda^3\R^8)$,
$(G_2,S^2_0\R^7)$.
The second and fourth representations admit enlargements to
$\Spin{16}$ (spin representation) and $\SO7$, respectively, which
are polar representations, hence they are not highly curved. 
The third representation
fails to satisfy~(\ref{real-bd1}). The first 
representation~$(\SO7,\Lambda^3\R^7)$ admits an isotropy group
$K=G_p\cong\SO2^3$ which is a maximal torus of $G$ (say,
$p=a\,e_1\wedge e_2\wedge e_3 + b\,e_3\wedge e_4\wedge e_5$ for generic
coefficients $a$, $b$). Now the fixed point set $W$ of $K$ is $3$-dimensional
and $H=N_G(K)/K$ is finite, so $Y=S(W)/H$ has constant curvature~$1$
and the existence of the folding map $I_{(K)}:Y\to X$ implies $\kappa_X=1$.

\subsubsection{Case~(b)}
In view of Lemmma~\ref{mainlem},
\begin{equation}\label{cx-bd}
  \dim G'+7\leq\dim G'+2+m=\dim V\leq 2\dim G' + 4 - \mathrm{rk}\,G'.
\end{equation}
In case $\mathrm{rk}\,G'=1$ we may assume $G'=\SU2$ and then~(\ref{cx-bd})
gives a contradiction. In case  $\mathrm{rk}\,G'\geq2$, (\ref{cx-bd})
gives $\dim V\leq2\dim G'+2$ and we can use ~\cite[Proposition]{K2}
to deduce that $\rho=(\U7,\Lambda^3\C^7)$. This representation
is not highly curved
because it can be enlarged
to~$(\SU8,\Lambda^4\C^8)$, which is a polar representation. 
Indeed, $\Lambda^3\C^7$ can be viewed as an $\U7$-invariant real form of
$\Lambda^4\C^8$ via
\[ x\in\Lambda^3\C^7\mapsto\frac12(x\wedge e_8+\epsilon(x\wedge e_8))\in\Lambda^4\C^8 \]
where $\epsilon$ is the Hodge star operator followed by complex conjugation
(see also~\cite[p.882]{yamatsu}). 

\subsubsection{Case~(c)}
In view of Lemmma~\ref{mainlem},
\begin{equation}\label{quat-bd}
  \dim V\leq 2\dim G' + 8 - \mathrm{rk}\,G'.
\end{equation}
In case $\mathrm{rk}\,G'=1$, we may assume $G=\SP1\times\SP1$ and
then~(\ref{quat-bd}) gives $V=\Q^3\otimes_{\mathbb H}\Q$.
This representation is covered by subsection~\ref{h3xh}.

In case  $\mathrm{rk}\,G'\geq2$, (\ref{quat-bd})
gives $\dim V\leq2\dim G'+6$ and we can use~\cite[Proposition]{K2}
to deduce that $\rho=(\SP1\times\Spin{11},\Q\otimes_{\mathbb H}\Q^{16})$ or
$(\SP1\times\Spin{13},\Q\otimes_{\mathbb H}\Q^{32})$. These representations
are not highly curved because they can be respectively enlarged to
$(\SP1\times\Spin{12},\Q\otimes_{\mathbb H}\Q^{16})$ and
$(\Spin{16},\R^{128})$, which are polar.

\subsection{The remaining case}
In view of subsection~\ref{simple},
there remains only to consider the case of irreducible representations
$\rho$ that can be decomposed as a tensor product $\rho_1\otimes\rho_2$
where $V=\mathbb F^2\otimes_{\mathbb F} V_2$, $G=G_1\times G_2$,
$s:=\dim_{\mathbb F}V_2\geq2$ and 
$\mathbb F$ can be~$\R$, $\C$ or~$\Q$.

Consider a pure tensor $v=v_1\otimes v_2$. 
Consider the enlargement to the 
polar representation of cohomogeneity $2$ of a compact connected
Lie group $H$. 
Consider a geodesic $\gamma$ 
starting at $v$ in a certain $H$-horizontal direction. It is then 
automatically $G$-horizontal, and we may choose it to contain $G$-regular 
points. Since the quotient
$S(V)/H$ is an interval of length $\pi/4$, the point $\gamma (\pi/2)$ is 
again on the same $H$-orbit, hence it is again a pure tensor.

\subsubsection{The case $V=\R^2\otimes_{\mathbb R} V_2$}\label{ss1}
Here $G_1=\SO2$ and $V_2$ is a representation 
of real type (since $V$ is irreducible). 

\begin{prop}
Assume that the action of $G_2$ on $S(V_2)$ 
has singular orbits. 
Then  $\kappa _{\rho} =1$.
\end{prop}

\begin{proof}
  We have $n=2s-1$. Since $\rho_2$ has singular orbits,
  for any $v_2$ in $S(V_2)$, the $G_2$-orbit through $v_2$ 
has dimension at most~$s-2$. Moreover, 
there is a point $w_2$ such that the orbit $G_2\cdot w_2$ has dimension
at most $s-3$. Thus a regular horizontal geodesic $\gamma$ as above
that starts at the pure tensor $w_1\otimes w_2$, starts at an orbit of 
dimension at most $s-2$ and intersects at time $\pi/2$ an orbit of dimension 
at most $s-1$. Since $(s-2)+(s-1)<2s-2=n-1$, we obtain $\kappa_\rho=1$ 
from Lemma~\ref{improvedlem}. 
\end{proof}

\medskip

It remains to consider the possibility
that $\rho_2$ has no singular orbits.  

In case $\rho_2$ is a representation of cohomogeneity one
of real type, we deduce from the classification 
that $\rho$ is either polar or has cohomogeneity three,
contrary to our assumptions. 

If $G_2$ acts non-transitively and
without singular orbits on $S(V_2)$, then it
is a (non-Abelian) group of rank $1$, but all representations
of real type of $\SO3$ admit singular orbits, so this case cannot occur.  

\subsubsection{The case $V=\Q^2\otimes_{\mathbb H} V_2$}\label{ss2}
We have $\rho(G)\subset\SP2\otimes\SP{V_2}$, where $V_2$ is
complex irreducible of quaternionic type
and~$s:=\dim_{\mathbb H}V_2\geq2$.  
We may also assume $G_1$ and $G_2$ are simple, for otherwise
we could rearrange the factors of $V$ 
and fall into case of a real tensor product. 

Since there does not exist a representation non-equivalent but
orbit-equivalent to $\SP2\times\SP s$, if $G_1\neq\SP2$ then 
we may enlarge $\rho$ to a 
representation $\hat\rho$ of $G_1\times\SP s$ 
and still have $\hat\rho$ of
cohomogeneity at least~$3$. Due to Proposition~\ref{enlargement},
it suffices to check that~$\hat\rho$ is not highly curved.  
Indeed this representation
is an enlargement of the doubling of the 
vector representation of $\SP s$.  The latter has
cohomogeneity~$6$.
Since the action of $G_1\times\SP s$ is clearly not orbit-equivalent 
to that of~$\SP s$, its orbit space has smaller dimension, 
and then we already know that~$\kappa_{\hat\rho}=1$. 

Otherwise $G_1=\SP2$ and  the group $G_2$ has rank at 
most $5$. If $g_2$ and $k_2$ denote the dimension and rank of $G_2$, resp.,
Lemma~\ref{mainlem} yields
$\dim_{\mathbb R} V_2=4s\leq g_2+\frac{21-k_2}2\leq g_2+10$.
Referring to~\cite[Table, p.~71]{GP},
we deduce that $\rho_2$ must be one of 
\[ (\SP1,\Q^2),\ (\Spin{11},\Q^{16}),\ (\SP1,\Q^3),\, (\SU6,\Lambda^3\C^6),\,
(\SP3,\Lambda^3_0\C^6). \]
In the first case, $\rho$ is a representation of cohomogeneity~$3$,
which is not highly curved. The second representation does not satisfy 
$4s\leq g_2+\frac{21-k_2}2$. In order to deal with the third 
representation, note that the maximal dimension of a $\SP2\times\SP1$-orbit 
through a pure tensor in~$\Q^2\otimes_{\mathbb H}\Q^2$ 
is~$7+3=10$, so we find a regular horizontal geodesic of length~$\pi/2$ 
which meets two orbits of dimension at most~$10$. 
Since $10+10<22=23-1$, we obtain~$\kappa=1$
from Lemma~\ref{improvedlem}.

To rule out the last two
representations, one can use the following proposition.

\begin{prop}\label{quat}
Assume that the action of $G_2$ on the quaternionic projective space 
$\Q P^{s-1}$ has an orbit of codimension at least $8$. Then $\kappa _{\rho} =1$.
\end{prop}

\begin{proof}
We have $n=8s-1$.  The dimension of the orbit through any pure 
tensor~$v_1\otimes v_2$ is at most $7+t$, where $t$ is the maximal 
dimension of the $G_2$-orbits on $\Q P^{s-1}$. Thus the dimension
of the $G$-orbit through $v_1\otimes v_2$ is at most $7+(4s-5)=4s+2$. 
Under the standing assumptions, we find a regular horizontal 
geodesic $\gamma$ of length~$\pi/2$ which meets an orbit of dimension at most 
$7+(4s-12)=4s-5$ and an orbit of dimension at most $4s+2$.
Since $(4s+2)+(4s-5)=8s-3<n-1$, we obtain $\kappa_\rho=1$ 
from Lemma~\ref{improvedlem}. 
\end{proof}

Note that the action of $G_2$ on~$\Q P^{s-1}$ has an orbit of 
codimension at least~$8$ if and only if the lift to an 
irreducible representation $\tilde\rho_2$ of $\SP1\times G_2$ on~$\Q^s$
has an orbit of codimension at least~$9$. The remaining
two cases for $\rho_2$ yield for $\tilde\rho_2$ the isotropy 
representations of the symmetric spaces $\E6/(\SU6\cdot\SU2)$ 
and $\F/(\SP3\cdot\SP1)$, whose restricted root systems 
have Coxeter type $\F$. The worst case for us is the second one,
in which all multiplicities are $1$. Corresponding to a subsystem 
of type~$\sf B_3$, we find a singular orbit of $\tilde\rho_2$ 
of codimension $4+9\cdot 1=13\geq9$, so Proposition~\ref{quat} applies. 

\subsubsection{The case $V=\C^2\otimes_{\mathbb C} V_2$}
We have $G_1=\U2$, $\rho(G)\subset\U2\otimes\SU{V_2}$ 
and~$s:=\dim_{\mathbb C}V_2\geq2$. We may assume that $G_2$ has no circle
factor and that $\rho_2$ 
is irreducible and of complex type. 

Similar to Proposition~\ref{quat}, one proves: 

\begin{prop}\label{cx}
Assume that the action of $G_2$ on the complex projective space 
$\C P^{s-1}$ has an orbit of codimension at least $4$. Then $\kappa _{\rho} =1$.
\end{prop}

Owing to Proposition~\ref{cx}, it remains only to 
discuss the case in which the action of $G_2$ on $\C P^{s-1}$
has all orbits of codimension at most $3$. Under this assuption,
that action lifts to an irreducible representation $\tilde\rho_2$
of $\U1\times G_2$ on $\C^s$ all of whose nonzero orbits have codimension
at most~$4$ and hence $\tilde\rho_2$ has cohomogeneity at most~$3$.

If the cohomogeneity of $\tilde\rho_2$ is one or two, then this is a 
polar representation whose restriction to the non $\U1$-factor remains 
irreducible. Going through the classification, we see that $\tilde\rho_2$ 
is one of the isotropy representations of the symmetric spaces:
\[ \SU{s+1}/\U s,\ \SU{2+\frac s2}/\sf S(\U2\times\U{\frac s2})\ 
\mbox{($\frac s2>2$)},\] \[\SO{10}/\U5,\ \E6/(\U1\cdot\Spin{10}). \] 
In the first case, $\rho$ is a polar representation so it is not highly curved. 
In the other cases, the restricted root system of the 
symmetric space has Coxeter type $\sf B_2$ with multiplicities
$(2,s-3)$, $(4,5)$ and~$(9,6)$, so we find an orbit of 
$\tilde\rho_2$ of codimension $2+s-3=s-1>4$, $2+5=7>4$, $2+9=11>4$.
This remark rules out all cases. 

If the cohomogeneity of $\tilde\rho_2$ is $3$, recall that 
$\rho_2$ is irreducible of complex type and $\mathrm{rk}\,G_2\leq5$, so
from the classification~\cite{hl,str} we get that $\tilde\rho_2$ is 
one of the isotropy representations of the symmetric spaces:
$\SP3/\U3$, $\SO{12}/\U6$, $\SU6/\sf{S}(\U3\times\U3)$
or~$\SU7/\sf{S}(\U3\times\U4)$. All symmetric spaces have Coxeter 
type~$\sf B_3$ and the worst case for us is $\SP3/\U3$ in which all 
multiplicities are $1$. In this case, 
corresponding to a subsystem of type $\sf B_2$,
we find a singular orbit of $\tilde\rho_2$ of codimension
$3+4\cdot 1=7>4$, which cannot be. This finishes the proof of the 
theorem.

\section{Appendix}

\noindent\(\pmb{\text{a1}[\text{t$\_$}]\text{:=}1+\frac{27}{(5+4 \text{Cos}[2 t])^2}}\)

\bigskip

\noindent\(\pmb{\text{a2}[\text{t$\_$}]\text{:=}1+\frac{27}{\left(-5+2 \text{Cos}[2 t]+2 \sqrt{3} \text{Sin}[2 t]\right)^2}}\)

\bigskip

\noindent\(\pmb{\text{a3}[\text{t$\_$}]\text{:=}1+\frac{27}{\left(5-2 \text{Cos}[2 t]+2 \sqrt{3} \text{Sin}[2 t]\right)^2}}\)

\bigskip

\noindent\(\pmb{\text{c1}[\text{t$\_$}]\text{:=}}\\
\pmb{1+}\\
\pmb{\left(27 (1-8 \text{Cos}[2 t]+2 \text{Cos}[4 t]-4 \text{Cos}[6 t])^2 \left(5-2 \text{Cos}[2 t]+2 \sqrt{3} \text{Sin}[2 t]\right)^2 \right.}\\
\pmb{\left.\left.\left(-5+2 \text{Cos}[2 t]+2 \sqrt{3} \text{Sin}[2 t]\right)^2\right)\right/}\\
\pmb{\left((5+4 \text{Cos}[2 t]) (21-20 \text{Cos}[2 t]+8 \text{Cos}[4 t])^2 \right.}\\
\pmb{\left(-10+2 \text{Cos}[2 t]-5 \text{Cos}[4 t]+4 \text{Cos}[6 t]-2 \sqrt{3} \text{Sin}[2 t]-5 \sqrt{3} \text{Sin}[4 t]\right) }\\
\pmb{\left.\left(-10+2 \text{Cos}[2 t]-5 \text{Cos}[4 t]+4 \text{Cos}[6 t]+2 \sqrt{3} \text{Sin}[2 t]+5 \sqrt{3} \text{Sin}[4 t]\right)\right)}\)

\bigskip

\noindent\(\pmb{\text{c2}[\text{t$\_$}]\text{:=}}\\
\pmb{1+\left.\left(27 \left(2 \sqrt{3}+\sqrt{3} \text{Cos}[2 t]+\text{Sin}[2 t]+4 \text{Sin}[4 t]\right)^2\right)\right/}\\
\pmb{\left((5+4 \text{Cos}[2 t]) \left(-5+2 \text{Cos}[2 t]+2 \sqrt{3} \text{Sin}[2 t]\right) \right.}\\
\pmb{\left.\left(-10+2 \text{Cos}[2 t]-5 \text{Cos}[4 t]+4 \text{Cos}[6 t]-2 \sqrt{3} \text{Sin}[2 t]-5 \sqrt{3} \text{Sin}[4 t]\right)\right)}\)

\bigskip

\noindent\(\pmb{\text{c3}[\text{t$\_$}]\text{:=}}\\
\pmb{1-\left.\left(27 \left(-2 \sqrt{3}-\sqrt{3} \text{Cos}[2 t]+\text{Sin}[2 t]+4 \text{Sin}[4 t]\right)^2\right)\right/}\\
\pmb{\left((5+4 \text{Cos}[2 t]) \left(5-2 \text{Cos}[2 t]+2 \sqrt{3} \text{Sin}[2 t]\right) \right.}\\
\pmb{\left.\left(-10+2 \text{Cos}[2 t]-5 \text{Cos}[4 t]+4 \text{Cos}[6 t]+2 \sqrt{3} \text{Sin}[2 t]+5 \sqrt{3} \text{Sin}[4 t]\right)\right)}\)

\bigskip

\noindent\(\pmb{\text{b1}[\text{t$\_$}]\text{:=}-\frac{648 (2-10 \text{Cos}[2 t]+2 \text{Cos}[4 t]-5 \text{Cos}[6 t]+2 \text{Cos}[8 t]) \text{Sin}[2
t]}{(5+4 \text{Cos}[2 t])^2 (21-20 \text{Cos}[2 t]+8 \text{Cos}[4 t])^2 \sqrt{\frac{\text{Sin}[6 t]^2}{65+16 \text{Cos}[6 t]}}}}\)

\bigskip

\noindent\(\pmb{\text{b2}[\text{t$\_$}]\text{:=}}\\
\pmb{-\left(\left.\left(324 (1+2 \text{Cos}[4 t]) \text{Sin}[2 t] \left(5 \text{Cos}[2 t]-2 \text{Cos}[4 t]+\sqrt{3} (5 \text{Sin}[2 t]+2 \text{Sin}[4
t])\right)\right)\right/\right.}\\
\pmb{\left.\left((5+4 \text{Cos}[2 t])^2 (21-20 \text{Cos}[2 t]+8 \text{Cos}[4 t])^2 \sqrt{\frac{\text{Sin}[6 t]^2}{65+16 \text{Cos}[6 t]}}\right)\right)}\)

\bigskip

\noindent\(\pmb{\text{b3}[\text{t$\_$}]\text{:=}\left.\left(324 (1+2 \text{Cos}[4 t]) \text{Sin}[2 t] \left(-5 \text{Cos}[2 t]+2 \text{Cos}[4 t]+\sqrt{3}
(5 \text{Sin}[2 t]+2 \text{Sin}[4 t])\right)\right)\right/}\\
\pmb{\left((5+4 \text{Cos}[2 t])^2 (21-20 \text{Cos}[2 t]+8 \text{Cos}[4 t])^2 \sqrt{\frac{\text{Sin}[6 t]^2}{65+16 \text{Cos}[6 t]}}\right)}\)

\bigskip

\noindent\(\pmb{\text{mu}[\text{t$\_$}]\text{:=}-27 \sqrt{\frac{(1-4 \text{Cos}[2 t])^2}{(5+4 \text{Cos}[2 t])^3 (21-20 \text{Cos}[2 t]+8 \text{Cos}[4
t])}}-}\\
\pmb{\frac{648 (2-10 \text{Cos}[2 t]+2 \text{Cos}[4 t]-5 \text{Cos}[6 t]+2 \text{Cos}[8 t]) \text{Sin}[2 t]}{(5+4 \text{Cos}[2 t])^2 (21-20 \text{Cos}[2
t]+8 \text{Cos}[4 t])^2 \sqrt{\frac{\text{Sin}[6 t]^2}{65+16 \text{Cos}[6 t]}}}}\)

\bigskip

\noindent\(\pmb{\text{E}[\text{t$\_$}]\text{:=}\frac{(-1+2 \text{Cos}[2 t]) (1+2 \text{Cos}[2 t])^2}{21-20 \text{Cos}[2 t]+8 \text{Cos}[4 t]}}\)

\bigskip

\noindent\(\pmb{\text{F}[\text{t$\_$}]\text{:=}\frac{(1+2 \text{Cos}[2 t]) \sqrt{\frac{(5+4 \text{Cos}[2 t]) (1+2 \text{Cos}[4 t])^2 \text{Csc}[2
t]^2}{21-20 \text{Cos}[2 t]+8 \text{Cos}[4 t]}} \text{Sin}[2 t]}{5+4 \text{Cos}[2 t]}}\)

\begin{figure*}[t!]
    \centering
    \begin{subfigure}[t]{0.5\textwidth}
        \centering
        \includegraphics[height=1.2in]{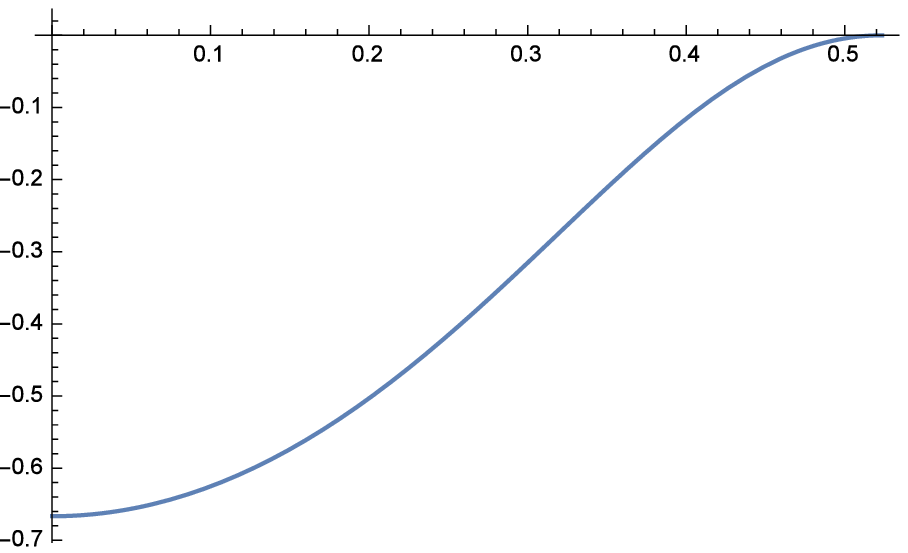}
        \caption{$C_1/\alpha_1$}
    \end{subfigure}%
    ~ 
    \begin{subfigure}[t]{0.5\textwidth}
        \centering
        \includegraphics[height=1.2in]{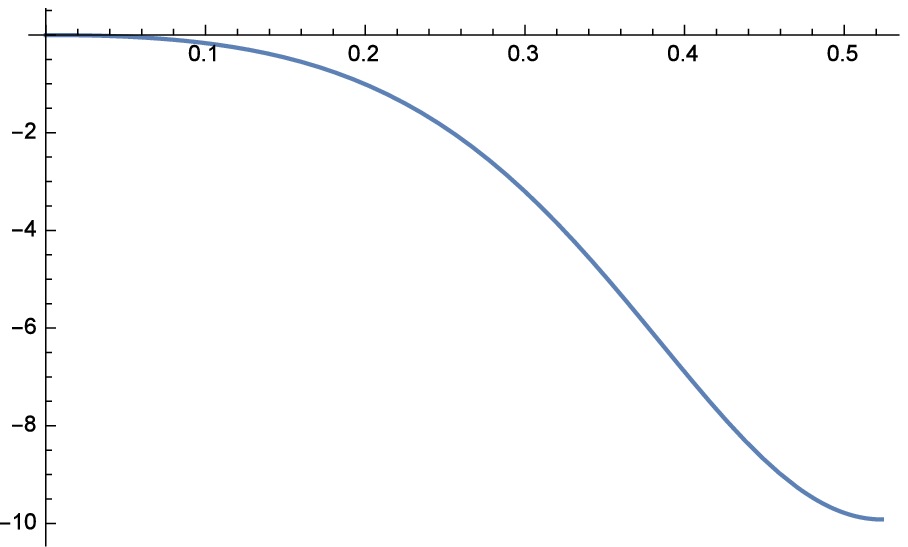}
        \caption{$C_2/\alpha_2$}
    \end{subfigure}
   \\
    \begin{subfigure}[t]{0.5\textwidth}
        \centering
        \includegraphics[height=1.2in]{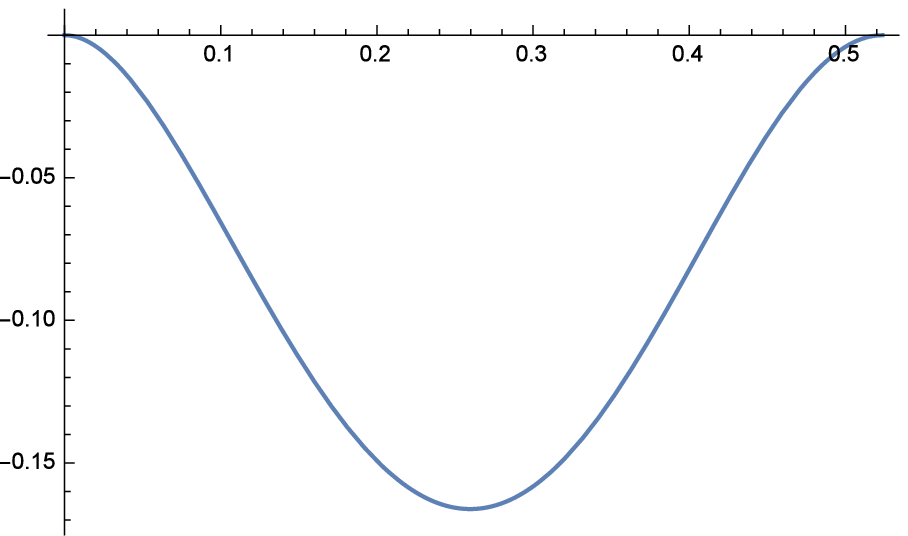}
        \caption{$C_3/\alpha_3$}
    \end{subfigure}
    \caption{Graphs of of...}\label{2}
\end{figure*}


\providecommand{\bysame}{\leavevmode\hbox to3em{\hrulefill}\thinspace}
\providecommand{\MR}{\relax\ifhmode\unskip\space\fi MR }
\providecommand{\MRhref}[2]{%
  \href{http://www.ams.org/mathscinet-getitem?mr=#1}{#2}
}
\providecommand{\href}[2]{#2}

\end{document}